\documentclass[reqno, 12pt]{amsart}
\pdfoutput=1
\makeatletter
\let\origsection=\section \def\section{\@ifstar{\origsection*}{\mysection}} 
\def\mysection{\@startsection{section}{1}\z@{.7\linespacing\@plus\linespacing}{.5\linespacing}{\normalfont\scshape\centering\S}}
\makeatother        

\usepackage{amsmath,amssymb,amsthm}
\usepackage{mathrsfs}
\usepackage{mathabx}\changenotsign
\usepackage{dsfont}
 
\usepackage{xcolor}
\usepackage[backref]{hyperref}
\hypersetup{
    colorlinks,
    linkcolor={red!60!black},
    citecolor={green!60!black},
    urlcolor={blue!60!black}
}

\usepackage[open,openlevel=2,atend]{bookmark}

\usepackage[abbrev,msc-links,backrefs]{amsrefs} 
\usepackage{doi}

\renewcommand{\PrintDOI}[1]{\doi{#1}}

\usepackage[T1]{fontenc}
\usepackage{lmodern}
\usepackage[babel]{microtype}

\usepackage[english]{babel}

\usepackage{geometry}
\usepackage{enumitem}
\def\rmlabel{\upshape({\itshape \roman*\,})}

\def\alabel{\upshape({\itshape \alph*\,})}

\let\polishlcross=\l
\def\l{\ifmmode\ell\else\polishlcross\fi}

\let\emptyset=\varnothing
\let\setminus=\smallsetminus

\makeatletter
\def\moverlay{\mathpalette\mov@rlay}
\def\mov@rlay#1#2{\leavevmode\vtop{   \baselineskip\z@skip \lineskiplimit-\maxdimen
   \ialign{\hfil$\m@th#1##$\hfil\cr#2\crcr}}}
\newcommand{\charfusion}[3][\mathord]{
    #1{\ifx#1\mathop\vphantom{#2}\fi
        \mathpalette\mov@rlay{#2\cr#3}
      }
    \ifx#1\mathop\expandafter\displaylimits\fi}
\makeatother

\newcommand{\dcup}{\charfusion[\mathbin]{\cup}{\cdot}}

\let\epsilon=\varepsilon
\let\eps=\epsilon
\let\rho=\varrho

\mathcode`l="8000
\begingroup
\makeatletter
\lccode`\~=`\l
\DeclareMathSymbol{\lsb@l}{\mathalpha}{letters}{`l}
\lowercase{\gdef~{\ifnum\the\mathgroup=\m@ne \ell \else \lsb@l \fi}}\endgroup

\usepackage{tikz}
\usetikzlibrary{calc,decorations.pathmorphing}
\pgfdeclarelayer{background}
\pgfdeclarelayer{foreground}
\pgfdeclarelayer{front}
\pgfsetlayers{background,main,foreground,front}

\def\EE{{\mathds E}}

\def\PP{{\mathds P}}

\newcommand{\cQ}{\mathcal{Q}}

\newcommand{\hedge}[7]{
		\coordinate (x) at #1;
		\coordinate (y) at #2;
		\coordinate (z) at #3;

		\ifx\relax#5\relax
			\def\lwidth{0}
		\else
			\def\lwidth{#5}
		\fi

		\def\cliplength{2*#4+2*\lwidth}
		\def\hlwidth{0.5*\lwidth}
		\coordinate (xy) at ($(x)!#4!90:(y)$);
		\coordinate (yx) at ($(y)!#4!-90:(x)$);
		\coordinate (yz) at ($(y)!#4!90:(z)$);
		\coordinate (zy) at ($(z)!#4!-90:(y)$);
		\coordinate (zx) at ($(z)!#4!90:(x)$);
		\coordinate (xz) at ($(x)!#4!-90:(z)$);
				\coordinate (xyi) at ($(x)!#4-\hlwidth!90:(y)$);
		\coordinate (yxi) at ($(y)!#4-\hlwidth!-90:(x)$);
		\coordinate (yzi) at ($(y)!#4-\hlwidth!90:(z)$);
		\coordinate (zyi) at ($(z)!#4-\hlwidth!-90:(y)$);
		\coordinate (zxi) at ($(z)!#4-\hlwidth!90:(x)$);
		\coordinate (xzi) at ($(x)!#4-\hlwidth!-90:(z)$);
		\coordinate (mx) at ($(xz)!.5!(xy)$);
		\coordinate (my) at ($(yx)!.5!(yz)$);
		\coordinate (mz) at ($(zy)!.5!(zx)$);
				
		\ifx\relax#7\relax
					\else
			\fill[#7] (xyi) -- (yxi) -- (yzi) -- (zyi) -- (zxi) -- (xzi) -- cycle;	
  	\fi

		\begin{scope}
    \clip ($(xz)!\hlwidth!(x)$) -- ($(x)!\cliplength!(xz)$) -- ($(x)!\cliplength!(mx)$) -- ($(x)!\cliplength!(xy)$) -- ($(xy)!\hlwidth!(x)$) -- cycle;
    	\ifx\relax#7\relax\else
	 			\fill[#7] (x) circle (#4);
			\fi
			\ifx\relax#5\relax\else
				\draw[line width=\lwidth,#6] (x) circle (#4);
			\fi
		\end{scope}
		
		\begin{scope}
    \clip ($(yx)!\hlwidth!(y)$) -- ($(y)!\cliplength!(yx)$) -- ($(y)!\cliplength!(my)$) -- ($(y)!\cliplength!(yz)$) -- ($(yz)!\hlwidth!(y)$) -- cycle;
 			\ifx\relax#7\relax\else
	 			\fill[#7] (y) circle (#4);
			\fi
			\ifx\relax#5\relax\else
				\draw[line width=\lwidth,#6] (y) circle (#4);
			\fi
		\end{scope}
		
		\begin{scope}
    \clip ($(zx)!\hlwidth!(z)$) -- ($(z)!\cliplength!(zx)$) -- ($(z)!\cliplength!(mz)$) -- ($(z)!\cliplength!(zy)$) -- ($(zy)!\hlwidth!(z)$) -- cycle;
 			\ifx\relax#7\relax\else
	 			\fill[#7] (z) circle (#4);
			\fi
			\ifx\relax#5\relax\else
				\draw[line width=\lwidth,#6] (z) circle (#4);
			\fi
		\end{scope}
		
		\ifx\relax#5\relax\else
			\draw[line width=\lwidth,line cap=round,#6] (xy) -- (yx);
			\draw[line width=\lwidth,line cap=round,#6] (yz) -- (zy);
			\draw[line width=\lwidth,line cap=round,#6] (zx) -- (xz);
		\fi
}

\newtheoremstyle{note}  {4pt}  {4pt}  {\sl}  {}  {\bfseries}  {.}  {.5em}          {}
\newtheoremstyle{introthms}  {3pt}  {3pt}  {\itshape}  {}  {\bfseries}  {.}  {.5em}          {\thmnote{#3}}
\newtheoremstyle{remark}  {2pt}  {2pt}  {\rm}  {}  {\bfseries}  {.}  {.3em}          {}

\theoremstyle{plain}
\newtheorem{theorem}{Theorem}[section]
\newtheorem{lemma}[theorem]{Lemma}

\newtheorem{conj}[theorem]{Conjecture}

\newtheorem{claim}[theorem]{Claim}

\theoremstyle{note}
\newtheorem{defin}[theorem]{Definition}

\theoremstyle{remark}
\newtheorem{remark}[theorem]{Remark}

\begin{document}
\title[Hamiltonicity in high degree triple systems]{On the Hamiltonicity of triple systems with high minimum degree}

\author[V.~R\"{o}dl]{Vojt\v{e}ch R\"{o}dl}
\address{Department of Mathematics and Computer Science,
Emory University, Atlanta, USA}
\email{rodl@mathcs.emory.edu}

\author[A.~Ruci\'nski]{Andrzej Ruci\'nski}
\address{A. Mickiewicz University, Department of Discrete Mathematics, Pozna\'n, Poland}
\email{rucinski@amu.edu.pl}

\author[M.~Schacht]{Mathias Schacht}
\address{Fachbereich Mathematik, Universit\"at Hamburg, Hamburg, Germany}
\email{schacht@math.uni-hamburg.de}

\author[E.~Szemer\'edi]{Endre Szemer\'edi}
\address{Alfr\'ed R\'enyi Institute of Mathematics,
	 Hungarian Academy of Sciences,
	 Budapest, Hungary}
\email{szemered@cs.rutgers.edu}

\thanks{V.~R\"odl was supported by NSF grant DMS 1301698. 
	A.~Ruci\'nski was supported by the Polish NSC grant 2014/15/B/ST1/01688.
	M.~Schacht was supported through the \emph{Heisenberg-Programme} of the DFG\@.
	Part of research was carried out during research stays of the A.~Ruci\'nski and  M.~Schacht
	at Emory University.}


\begin{abstract}
We show that every  3-uniform hypergraph  with minimum vertex degree at least
$0.8\tbinom{n-1}2$ contains a tight Hamiltonian cycle.
\end{abstract}

\maketitle

\section{Introduction}\label{intro}
In 1952 Dirac \cite{dirac} proved that every graph $G=(V,E)$ with  $|V|\geq 3$ and minimum vertex degree $\delta(G)$ 
at least $|V|/2$ 
contains a Hamiltonian cycle. Moreover, this is optimal as there are graphs~$G$ with $\delta(G)=\lceil |V|/2\rceil-1$ not containing a Hamiltonian cycle. We study an analogous Dirac-type problem for $3$-uniform hypergraph, i.e.,  
what minimum vertex degree in a $3$-uniform hypergraph 
guarantees the existence of a (tight) Hamiltonian cycle? A lot of recent research 
concerning Dirac-type problems for hypergraphs originated in the work 
of Katona and Kierstead~\cite{KK} (see also~\cite{sur} for an overview).

A \emph{$k$-uniform hypergraph} $H=(V,E)$ (or \emph{$k$-graph}) consists of  a finite set $V=V(H)$
of \emph{vertices} together with a family $E=E(H)$ of $k$-element subsets 
of $V$, the so-called \emph{(hyper)edges}.
Whenever convenient we  identify $H$ with $E(H)$. In particular, we denote by $|H|=|E(H)|$ the number of edges in $H$.
For $k$-graphs with $k\ge3$,  a  cycle might be defined in several ways (see, e.g., \cites{berge,Bermond,KK,kob}). 
Here we restrict ourselves to $k=3$. For $l=1$ or $2$ and an integer~$n$ with 
$(3-l)|n$,  define an  \emph{$l$-overlapping cycle $C^l_n$}
 as an $n$-vertex $3$-graph  with $\frac n{3-l}$ edges, whose vertices
can be ordered cyclically in such a way that the edges are segments of that
cyclic ordering and every two consecutive edges share exactly~$l$ vertices.
 For $l=2$, we call the cycle \emph{tight} and for $l=1$ we call it \emph{loose}.
  A \emph{tight (resp.\ loose) Hamiltonian cycle} in a $3$-graph~$H$ is a spanning tight (resp.\ loose) cycle in~$H$, 
that is, a subhypergraph of~$H$ which is a tight (resp.\ loose) cycle that contains all vertices of $H$.

For a 3-graph $H=(V,E)$, in addition to having two types of cycles, there are also two natural 
notions of minimum vertex degree (see $\delta_1(H)$ and $\delta_2(H)$ below). 
For a vertex~$v\in V$ we  define $\deg_H(v)$ as the number of edges of $H$ containing $v$ and for every 
pair of distinct vertices $u$, $v\in V$ we define the \emph{co-degree/pair degree} of that pair, $\deg_H(u,v)$, by 
the number of edges of $H$ containing both $u$ and $v$.
Clearly, for an $n$-vertex 3-graph we have $\deg_H(v)\le \binom {n-1}2$ while $\deg_H(u,v)\le n-2$.
For a 3-graph $H=(V,E)$ we denote by 
$$
	\delta_1(H)=\delta(H)=\min_{v\in V}\deg_H(v)
$$ 
the \emph{minimum vertex degree of $H$} and by 
$$
	\delta_2(H)=\min_{\substack{u,\,v\,\in V\\ u\neq v}}\deg_H(u,v)
$$
the \emph{minimum co-degree of $H$}. 
 
We are now ready to define a crucial Dirac-type extremal parameter.
\begin{defin}
Let $d$, $l$, and $n$ be integers satisfying $1\le\l, d\le2$, and $(3-l)|n$.
The function $h^l_d(n)$ equals the smallest integer $h$ such that every $n$-vertex
$3$-graph $H$ with $\delta_d(H)\ge h$ contains a spanning $l$-overlapping cycle, that is, a loose Hamiltonian cycle for $l=1$ and 
a tight Hamiltonian cycle for $l=2$.  
In other words,
$$
h^l_d(n)=\min\big\{h\in[n]\colon \delta_d(H)\ge h\ \Longrightarrow\ H\supseteq C^l_n\big\}.
$$
\end{defin}

For large~$n$ and the following choices of $d$ and $l$ the function $h^l_d(n)$ is well understood.
The case $d=l=2$ (co-degree forcing Hamiltonian tight cycles) was solved approximatively and exactly (for large~$n$)
in~\cites{rrs3,3}, while the case $d=2$ and $l=1$ (co-degree forcing Hamiltonian loose cycles) was solved 
approximatively in~\cite{kob}. In~\cite{Bus} an approximate formula for $h_1^1(n)$ (vertex degree forcing Hamiltonian loose cycles)
was found, while an exact form of this result was obtained in~\cite{YiJie}. 
The related problem concerning minimum degree conditions for perfect matchings was resolved for co-degrees approximately and exactly in~\cites{rrs3,rrs} and 
similarly for vertex degrees in~\cites{hps,kot,khan}.

In particular, the results mentioned above resolve the asymptotic behaviour for all possible values of $d$ and $\l$
with the exception of $d=1$ and $l=2$. It seems that more difficulties arise in that case, since $d<l$
and, hence, we are not in control of co-degrees, while it seems that large co-degrees 
are instrumental in building long tight paths and cycles.
We will derive new bounds for $h_1^2(n)$ and for simplicity we set
\[
	h(n)=h_1^2(n)\,.
\]
Some estimates on $h(n)$ were obtained over the last few years.
While proving a more general result, Glebov, Person, and Weps \cite{GPW} showed that 
\[
	h(n)\le(1-\epsilon)\binom{n-1}2,
\]
where the numerical value of $\epsilon$ is close to
$5\times 10^{-7}$. In \cite{1112} the first two authors improved upon that bound by showing that
for every $\gamma>0$ there exists $n_0$ such that if $n\ge n_0$ then
$$h(n)\le \left(\frac{5-\sqrt5}{3}+\gamma\right)\binom  {n-1}2\approx .92\binom{n-1}2.$$
Here we make a further improvement.

\begin{theorem}\label{main}
 There exists $n_{\ref{main}}$ such that if $n\ge n_{\ref{main}}$ then
$$h(n)\le .8\binom  {n-1}2.$$
\end{theorem}

This upper bound on $h(n)$ seems to be far from optimal.
Indeed, the best known constructions  yield
\[
	h(n)\geq \left(\frac59+o(1)\right)\binom{n-1}2
\]
and we briefly mention three constructions achieving this bound.
\begin{enumerate}[label=\rmlabel]
	\item Consider a partition $X\dcup Y=V$ of the vertex set $V$ of size $n$ 
		with $|X|=\lceil (n+1)/3\rceil$ and let $H$ be the $3$-graph containing all 
		edges $e$ such that $|e\cap X|\neq 2$. It is not hard to show that~$H$ contains no tight 
		Hamiltonian cycle, since two consecutive vertices in $X$ cannot be connected to $Y$ (see, e.g.,~\cite{1112}). 
		Moreover, we have  $\delta(H)\geq (5/9+o(1))\binom{n-1}{2}$.
	\item Similarly, one may consider a partition $X\dcup Y=V$ with $|X|=\lceil 2n/3\rceil$
		and let~$H$ be the $3$-graph consisting of all 
		hyperedges $e$ such that $|e\cap X|\neq 2$. Again $H$ has $\delta(H)\geq (5/9+o(1))\binom{n-1}{2}$
		and it contains no tight Hamiltonian cycle.
	\item The last example utilises the fact that every tight Hamiltonian cycle contains a matching of 
		size~$\lfloor n/3\rfloor$. Again we consider a partition $X\dcup Y=V$ this time with $|X|=\lfloor n/3\rfloor-1$
		and let $H$ consist of all hyperedges having at least one vertex in~$X$. Consequently,~$H$ 
		contains no matching of size $\lfloor n/3\rfloor$ and, hence, no tight Hamiltonian cycle.
		On the other hand, $\delta(H)\geq (5/9+o(1))\binom{n-1}{2}$.
\end{enumerate}

It might be possible that these constructions give the right asymptotic lower bound for~$h(n)$, 
which leads to the following conjecture.

\begin{conj}\label{con59}
$
	h(n)=\left(\frac59+o(1)\right)\binom{n-1}2.
$
\end{conj}
However, we remark that recently Han and Zhao \cite{YiJieCounter} showed
that for $k\ge4$ some Dirac-type thresholds for tight Hamiltonian cycle 
are strictly larger than the corresponding thresholds for perfect matchings, which 
may put some doubt on Conjecture~\ref{con59}. However, it seems unlikely that the upper bound 
given Theorem~\ref{main} is optimal and we shall return to the problem of determining the asymptotic 
behaviour of~$h(n)$ in the near future.

\section{Outline of proof and preliminaries}

\subsection{Outline}

Our proof follows the absorbing path method developed in \cites{rrs3,k,3}. We begin with building an absorbing path $A$ and putting aside a small
reservoir set $R$ selected randomly so that $H[R]$ preserves the degree properties of $H$.  Then a long
cycle $C$ containing $A$ is created in the remaining hypergraph (by first building a family of disjoint paths and then connecting them, as well as $A$,  together via the reservoir $R$). Finally, utilising the absorbing property of $A$, the cycle $C$ is
extended to a Hamiltonian cycle in $H$.

Our proof is founded on four pillars: the Connecting Lemma, the Absorbing Lemma, the Reservoir Lemma, and the  Cover Lemma (replacing the Path Cover Lemma used earlier in~\cites{k,1112}), and we will prove them all in the next section. The Connecting Lemma and the Absorbing Lemma are the bottlenecks here. In \cite{1112} we `shaved' the hypergraph $H$ from edges containing pairs of small degree until all pairs of positive degree were large, that is, of degree a little bigger than $n/2$. In the obtained subhypergraph $H'$ proving a connecting lemma was easy, however we paid a high price for that: to prevent $H'$ from becoming empty, we had to raise the minimum vertex degree of $H$ to about $.92\binom{n-1}2$. Here we refine that approach: we only dispose of the edges of $H$ with all three pairs of small degree and, at the same time, we lower the notion of ``small'' to only $n/3$. Then both, the Connecting Lemma and the Absorbing Lemma, are  a bit harder to prove, yet we manage to do so, keeping~$\delta(H)$ at around $.8\binom{n-1}2$.
The Reservoir Lemma, as usual, can be proved by a standard application of the probabilistic method. Finally, the proof of the  Cover Lemma follows the lines of the  approach from \cites{k,1112} in that it relies on the Weak Regularity Lemma.
Once the four lemmas are proved, the actual proof of Theorem~\ref{main} consists of five simple steps (stated below). For any $S\subset V(H)$, let $H-S$ denote the induced subhypergraph $H[V(H)\setminus S]$, that is, a subhypergraph obtained from~$H$ by deleting all vertices in  $S$ together with the edges they belong to.

\begin{enumerate}
\item Find an absorbing path $A$ in $H$.
\item Find a reservoir set $R$ in $H- V(A)$.
\item Applying the  Cover Lemma to a suitable selected sub-3-graph $H'$ of $H$, find a collection of disjoint paths $\mathcal P$, covering most of the vertices of $H-(V(A)\cup R)$.
\item Connect the paths in $\mathcal P$ and the absorbing path $A$, using vertices of $R$, to form a cycle $C$ which contains most of the vertices of $H$.
    \item Using the absorbing property of $A$, put all the remaining vertices on the cycle to form a Hamiltonian cycle in $H$.
\end{enumerate}

\subsection{Preliminaries}

Here we collect basic tools needed in the subsequent proofs. We begin with a lower bound on the number of triangles in an $n$-vertex graph  in terms of the number of its edges.
Although more refined results are available (see Razborov~\cite{raz}), for us it will be sufficient to use an old  bound of Nordhaus and Stewart \cite{ns} which is also attributed to Goodman~\cite{Good} and Moon and Moser~\cite{MM} (see~\cite{BB}*{Corollary 1.6 in Chapter~VI}).

\begin{lemma}\label{NS}
Every graph $G$ with $n$ vertices and $m$ edges contains at least $\frac{m}{3n}\left(4m-n^2\right)$ triangles. In particular, if for some $\rho>0$ we have  $m\ge\rho\tfrac{n^2}2$ then the number of triangles in $G$ is at least $\rho(2\rho-1)\tfrac{n^3}6$.
\end{lemma}

  We will also need the following version of a result of Erd\H os~\cite{E}.
A $3$-graph $H$ is \emph{$3$-partite} if there is a partition $V(H)=V_1\dcup V_2\dcup V_3$ such that
every edge of $H$ intersects each set~$V_i$ in precisely one vertex.
A 3-partite 3-graph with $|V_1||V_2||V_3|$ edges is called \emph{complete} and denoted by $K_{h_1,h_2,h_3}$, where $h_i=|V_i|$, $i=1,2,3.$.
\begin{lemma}\label{erdos}
For every $d>0$ and an integer $h\ge1$, there exist $c>0$ and $n_{\ref{erdos}}$ such that  every $3$-uniform hypergraph $H$
on $n\ge n_{\ref{erdos}}$ vertices and  with at least $d n^3$ edges contains at least $cn^{3h}$ copies of $K_{h,h,h}$.
\end{lemma}

In the proof of the  Cover Lemma
we will also need the so-called \emph{weak hypergraph regularity lemma}, which is the straightforward
extension of Szemer\'edi's regularity lemma~\cite{Sz} from graphs to hypergraphs (see, e.g.~\cites{Ch91,FR92,St90}).

Given a $3$-graph $H$ and three non-empty, disjoint subsets $A_i\subset V(H)$, $i=1,2,3$, by~$H[A_1,A_2,A_3]$  we
denote the 3-partite 3-graph with vertex set $A_1\cup A_2\cup A_3$ which consists of all edges in $H$ with one vertex in each $A_i$.
We set $e_H(A_1,A_2,A_3)$ for the number of edges of $H[A_1,A_2,A_3]$ and define
\emph{ the density} of $H$ with respect to $(A_1,A_2,A_3)$ as
 $$d_H(A_1,A_2,A_3)=\frac{e_H(A_1,A_2,A_3)}{|A_1||A_2||A_3|}.$$
We say that a  3-partite 3-graph $H$ with 3-partition $(V_1,V_2,V_3)$ is \emph{ $\varepsilon$-regular}
 if for all $A_i\subseteq V_i$ with $|A_i|\ge\varepsilon|V_i|$, $i=1,2,3$,
$$|d_H(A_1,A_2,A_3)-d_H(V_1,V_2,V_3)|\le\varepsilon.$$

\begin{lemma}[Weak Regularity Lemma for 3-graphs]\label{WRL}
For all $\varepsilon>0$ and every integer $t_0$ there exist $T_0$ and $n_{\ref{WRL}}$ such that the following
 holds. For every $3$-graph $H$ on $n\ge n_{\ref{WRL}}$ vertices there is for some $t$,  with  $t_0\le t\le T_0$, a partition
 $V(H)=V_1\dcup\cdots\dcup V_t$ such that $|V_1|\le|V_2|\le\cdots\le|V_t|\le|V_1|+1$
 and for all but less than $\varepsilon\binom t3$ triplets of  partition classes $\{V_{i_1},V_{i_2},V_{i_3}\}$, the
the 3-partite 3-graph $H[V_{i_1},V_{i_2},V_{i_3}]$  is $\varepsilon$-regular.
\end{lemma}
\noindent Any partition guaranteed by Lemma \ref{WRL} will be referred to as \emph{$\eps$-regular}.

For brevity, we will often write $uv$ instead $\{u,v\}$. \emph{The link graph} of a vertex $u$ in a 3-graph $H$ is defined as
$$H(u)=\big\{vw\colon\{u,v,w\}\in H\big\}.$$
Note that
for every $v\neq u$
 \begin{equation}\label{degdeg}
 \deg_{H(u)}(v)= \deg_{H}(u,v).
 \end{equation}
 For each real $\alpha\in(0,1)$ we define
$$G_{\alpha}=\big\{uv\colon \deg_H(u,v)\ge\alpha(n-2)\big\}$$
and call a pair  \emph{$\alpha$-large} if it is in $G_\alpha$. The $1/3$-large pairs play a special role in our proof.
However, also $G_{.33}$ will appear in our proof and should not be confused with $G_{1/3}$.
Let
\begin{equation}\label{H0} H_0=\left\{e\in H\colon \binom e2\cap G_{1/3}=\emptyset\right\}\quad\mbox{ and }\quad H'=H\setminus H_0,
\end{equation}
that is, $H'$ is a spanning subhypergraph of $H$ with all edges of $H_0$ removed.
Note that every edge of $H'$ contains at least one pair from $G_{1/3}$.

We  build a tight Hamiltonian cycle in $H$ from several small pieces. \emph{Tight paths}
are defined in the same way as tight cycles, but with respect to a \emph{linear} ordering
of the vertices.
From now on we will refer to tight paths
and cycles as paths and cycles, resp.   If $P$ is a path with $t\ge3$ vertices $v_1,\dots, v_t$ and
$t-2$ edges $\{v_1,v_2,v_3\},\dots,\{v_{t-2},v_{t-1},v_t\}$, then we call the ordered pairs
$(v_1,v_2)$ and $(v_t,v_{t-1})$ the \emph{endpairs} of $P$, and we say that $P$ \emph{connects} its
endpairs. The \emph{length} of a path is defined as the
number of its edges and the \emph{order} denotes its number of vertices.

\section{The four pillars}

In this section we prove the four crucial lemmas:  the Connecting Lemma, the Absorbing Lemma, the Reservoir Lemma, and the  Cover Lemma.

\subsection{The Connecting Lemma}

 The connecting lemma in \cite{k} assumes that $\delta_2(H)$ is large and  guarantees a short path between any two ordered pairs of vertices. There is no hope for such a result here, as some pairs may have very small degree, even zero. So, we must be content with connecting just the pairs with large degrees. As a first step we establish  a numerical relation between $\delta(H)$ and $\delta(G_\alpha)$.
 To this end, for all $0<\alpha<c<1$, define
$$g_c(\alpha)=\frac{c-\alpha}{1-\alpha}.$$

\begin{claim}\label{f} Let $0<\alpha<c<1$.
If $\delta(H)\ge c\binom {n-1}2$ then $\delta(G_{\alpha})\ge g_c(\alpha)(n-1)$.
\end{claim}

\begin{proof} Set $G:=G_{\alpha}$.
 Let $u_0\in V(H)$ satisfy $\deg_G(u_0)=\delta(G)$. Then, by (\ref{degdeg}) we have
$$2\delta(H)\le2|H(u_0)|=\sum_{u\neq u_0}\deg_{H(u_0)}(u)= \sum_{u\neq u_0}\deg_{H}(u,u_0).$$
Breaking the latter sum into two parts: over $uu_0\in G$ and over $uu_0\notin G$, and recalling that $|\{u\colon uu_0\in G\}|=\delta(G)$, we obtain the  inequality
$$c(n-1)(n-2) \le 2\delta(H)\le\sum_{u\neq u_0}\deg_{H}(u,u_0)\le\delta(G)(n-2)+(n-1-\delta(G))\alpha(n-2), $$
from which the required bound follows. 
\end{proof}

We also  need a simple combinatorial inequality which was  observed already in \cite{1112}*{Fact~1}). \begin{claim}\label{BR} For any two finite sets $B$ and $R$, with $|B|\le |R|$, the set
$$
\Pi(B,R)
=\left\{\{b,r\}\colon b\in B, r\in R\right\}
=\left\{e\in\tbinom{B\cup R}2\colon e\cap B\neq\emptyset\mbox{ and }e\cap R\neq\emptyset\right\}$$
has size
$$|\Pi(B,R)|\ge \binom{|B|}2.$$
\end{claim}
\begin{proof} Let  $c=|B\cap R|$.  Then, as $|R|\ge|B|\ge c$,
\begin{multline*}
|\Pi(B,R)|
	=\binom{|B|+|R|-c}2-\binom{|B|-c}2-\binom{|R|-c}2\\
	=|B|\cdot|R|-\binom{c+1}2\ge{|B|}^2-\binom{|B|+1}2=\binom{|B|}2.\tag*{\qedhere}
\end{multline*}
\end{proof}

We are  now ready to prove the Connecting Lemma.

\begin{lemma}[Connecting Lemma]\label{conn}
There exists $n_{\ref{conn}}$ such that for all $n\ge n_{\ref{conn}}$  the following holds. Let $H$ be an $n$-vertex 3-graph with   $\delta(H)\ge.799\binom {n-1}2$.
Then, for all $e=u_0u_1\in G_{.33}$ and  $f=v_0v_1\in G_{.33}$ with $e\cap f=\emptyset$ there exists a path in $H$
of length 12 connecting the endpairs $(u_0,u_1)$ and $(v_0,v_1)$,
\end{lemma}

\begin{proof} We build a connecting path by  making our way from each side, increasing the degrees of the pairs as we go, until they both reach $.65(n-2)$, a quantity that guaranties an immediate connection of the two paths. In doing so we will use a sequence of numbers $\alpha_1,\dots,\alpha_5$ such that $\alpha_1=.33$, $\alpha_5=.65$, and for each $i=2,\dots,5$, with $c=.799$, we have
$$\alpha_i+g_{c}(\alpha_{i+1})>1.$$
By inspection we found that $\alpha_2=.39, \alpha_3=.48$, and $\alpha_4=.58$ satisfy these requirements.
 Below we use Claim \ref{f} repeatedly. For brevity, we write $g(\alpha)$ for $g_{.799}(\alpha)$.
 Moreover, for a graph 
  $G$ and a vertex $v$ we denote by $N_G(v)$ the neighbourhood of $v$ in $G$ and, similarly, we denote by 
  $N_H(u,v)=\{x\in V(H)\colon \{x,u,v\}\in E(H)\}$ the neighbours of the pair $uv$ in the $3$-graph~$H$. 
 \begin{itemize}
 \item Observe that, in view of Claim \ref{f},
 $$|N_{G_{.39}}(u_1)|\ge \delta(G_{.39})\ge g(.39)(n-1)=\tfrac{409}{610}(n-1)\ge.67(n-2)+20,$$
 where the last inequality holds for sufficiently large $n$.
Thus,
$$|N_H(u_0,u_1)|+|N_{G_{.39}}(u_1)|\ge .33(n-2)+.67(n-1)+20>n+18,$$
implying that
$$|N_H(u_0,u_1)\cap N_{G_{.39}}(u_1)|\ge|N_H(u_0u_1)|+|N_{G_{.39}}(u_1)|-(n-1)\ge20.$$
Consequently,
 there exists a vertex $u_2\not\in\{v_0,v_1\}$ such that $\{u_0,u_1,u_2\}\in H$ and $u_1u_2\in G_{.39}$.

 \item Next, since
 $$\delta(G_{.48})\ge g(.48)(n-1)=\tfrac{319}{520}(n-1)\ge.61(n-2)+20,$$
a
similar argument yields that
$$|N_H(u_1,u_2)\cap N_{G_{.39}}(u_2)|\ge20,$$
implying
the existence
 of  a vertex $u_3\not\in\{u_0,v_0,v_1\}$ such that $\{u_1,u_2,u_3\}\in H$ and $u_2u_3\in G_{.48}$.

 \item Analogously, since
 $$\delta(G_{.58})\ge g(.58)(n-1)=\tfrac{219}{420}(n-1)\ge.52(n-2)+20,$$
  there exists a vertex $u_4\not\in\{ u_0,u_1,v_0,v_1\}$ such that $\{u_2,u_3,u_4\}\in H$ and \newline$u_3u_4\in G_{.58}$.

 \item Finally, since
 $$\delta(G_{.65})\ge g(.65)(n-1)=\tfrac{149}{350}(n-1)\ge.42(n-2)+20,$$
  there exists a vertex $u_5\not\in\{u_3,u_4,v_0,v_1\}$ such that $\{u_3,u_4,u_5\}\in H$ and \newline$u_4u_5\in G_{.65}$.
 \end{itemize}
Hence, we have created a 4-edge path $P_u=u_0u_1\dots u_5$ in $H$ with $V(P_u)\cap f=\emptyset$ and $u_4u_5\in G_{.65}$.

 In a similar fashion we build a path $P_v=v_0\dots v_5$ which avoids all vertices of $P_u$ and such that also $\{v_4,v_5\}\in G_{.65}$.
 The additional $+20$ guarantees, with a margin,  that even  when choosing the last vertex, $v_5$, we can still avoid the already selected vertices.
To connect the two paths together, let us consider the intersection of link graphs of $u_5$ and $v_5$
$$I=H(u_5)\cap H(v_5).$$
Owing to the assumption $\delta(H)\ge.799\binom {n-1}2$ we have
\begin{equation}\label{EL}
|I|\ge|H(u_5)|+|H(v_5)|-|H(u_5)\cup H(v_5)|\ge2\delta(H)-\binom n2\ge .598\binom n2+O(n).
\end{equation}
Set
$$B=N_H(u_4,u_5)\quad\mbox{ and }\quad R=N_H(v_4,v_5)$$ and assume as we may
that $|B|\le |R|$. Let $F$ be the set of pairs $xy\in \binom{V(H)}2$ such that $\{u_4,u_5,x\}\in H$ and $\{v_4,v_5,y\}\in H$, that is,  the set of pairs of vertices  with one vertex belonging to $B$ and the other to $R$.
By Claim \ref{BR}, $|F|\ge \binom{|B|}2$ and, since $u_4u_5\in G_{.65}$, we have
$|B|\ge.65(n-2).$ Consequently,
$$|F|\ge\binom{|B|}2\ge (.65)^2\binom n2+O(n).$$
Since $.598+(.65)^2>1.02$, it follows from (\ref{EL}) and the above bound that
$$|F\cap I|\ge|F|+|I|-\binom n2>.02\binom n2+O(n),$$
which, for sufficiently large $n$, is greater than $8n$, which is an upper bound on the number of pairs  $\{x,y\}\in F\cap I$ with $\{x,y\}\cap V(P_u\cup P_v)\neq \emptyset$.
We conclude that there exist two vertices $x=u_6$ and $y=v_6$ different from all $u_0,\dots,u_5,v_0\dots,v_5$, such that all four triples $\{u_4,u_5,u_6\},\{u_5,u_6,v_6\}$, $\{u_6,v_6,v_5\},\{v_6,v_5,v_4\}$ are edges of $H$. Hence, a path between the endpairs $(u_0,u_1)$ and $(v_0,v_1)$ of length 12 can be completed (see Figure~\ref{fig:1}).
\end{proof}

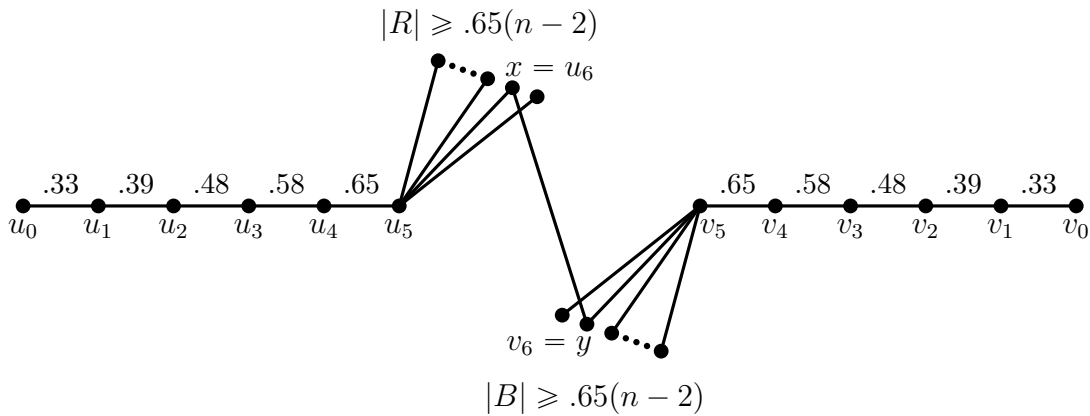
\begin{figure}
\begin{center}
\begin{tikzpicture} 
  	\foreach \x in {0,1,...,5}{
 		\fill (\x*1,0) circle (2.8pt); 
		\node at (\x*1,-0.3) {$u_\x$};		
	}
	\draw[very thick] (0,0) -- (5,0);	
	\node at (0.5,0.3) {\small$.33$};
	\node at (1.5,0.3) {\small$.39$};
	\node at (2.5,0.3) {\small$.48$};
	\node at (3.5,0.3) {\small$.58$};
	\node at (4.5,0.3) {\small$.65$};
	
	\coordinate (r1) at ($(5,0)+(75:2)$);
	\coordinate (r2a) at ($(r1)+(-20:0.2)$);
	\coordinate (r2b) at ($(r2a)+(-20:0.15)$);
	\coordinate (r2c) at ($(r2b)+(-20:0.15)$);
	\coordinate (r2) at ($(r2c)+(-20:0.2)$);
	\coordinate (r3) at ($(r2)+(-20:0.35)$);
	\coordinate (r4) at ($(r3)+(-20:0.35)$);
	
	\foreach \x in {1,2,...,4}{
		\fill (r\x) circle (2.8pt);
		\draw[very thick] (5,0) -- (r\x);
	}
	
	\fill (r2a) circle (1.2pt);
	\fill (r2b) circle (1.2pt);
	\fill (r2c) circle (1.2pt);
	
	\node at (6.2,2.4) {$|R|\geq .65(n-2)$};

	\foreach \x in {0,1,...,4}{
 		\fill (14-\x*1,0) circle (2.8pt); 
		\node at (14-\x*1,-0.3) {$v_\x$};		
	}
	\fill (9,0) circle (2.8pt);
	\node at (9.18,-0.3) {$v_5$};
	
	\draw[very thick] (14,0) -- (9,0);
	\node at (13.5,0.3) {\small$.33$};
	\node at (12.5,0.3) {\small$.39$};
	\node at (11.5,0.3) {\small$.48$};
	\node at (10.5,0.3) {\small$.58$};
	\node at (9.5,0.3) {\small$.65$};
	
	\coordinate (b1) at ($(9,0)+(255:2)$);
	\coordinate (b2a) at ($(b1)+(160:0.2)$);
	\coordinate (b2b) at ($(b2a)+(160:0.15)$);
	\coordinate (b2c) at ($(b2b)+(160:0.15)$);
	\coordinate (b2) at ($(b2c)+(160:0.2)$);
	\coordinate (b3) at ($(b2)+(160:0.35)$);
	\coordinate (b4) at ($(b3)+(160:0.35)$);

	\foreach \x in {1,...,4}{
		\fill (b\x) circle (2.8pt);
		\draw[very thick] (9,0) -- (b\x);
	}
	
	\fill (b2a) circle (1.2pt);
	\fill (b2b) circle (1.2pt);
	\fill (b2c) circle (1.2pt);
	
	\node at (7.6,-2.55) {$|B|\geq .65(n-2)$};
	
	\draw[very thick] (b3) -- (r3);
	\node at ($(r3)+(0.5,0.22)$) {$x=u_6$};
	\node at ($(b3)+(-0.5,-0.3)$) {$v_6=y$};
\end{tikzpicture}
\end{center}
\caption{Underlying pairs of the tight path of length 12 in $H$ connecting the pairs $u_0u_1$ with $v_1v_0$ from $G_{.33}$.}
\label{fig:1}
\end{figure}

\subsection{The Absorbing Lemma}

We first define absorbing path and introduce the notion of an absorber in this context.

\begin{defin}\label{AP} We call a path $A$ in $H$ \emph{$m$-absorbing} if
 for every subset
$U\subset V(H)\setminus V(A)$ of size  $|U|\le m$
 there is a path $A_U$ in $H$  with
$V(A_U)=V(A)\cup U$ and with the same  endpairs as $A$.
\end{defin}

Recall that a pair of vertices in $H$ is called  \emph{$\alpha$-large} if it belongs to $G_{\alpha}$.

\begin{defin}\label{XA}
 Given a vertex $x\in V(H)$, an \emph{$x$-absorber of order $i$} is a path $P=v_1\dots v_i$ in $H$ such that the graph path $Q=v_1\dots v_i$ is a subgraph of the link graph $H(x)$ and the  endpairs $v_1v_2$ and $v_{i}v_{i-1}$ of $P$ belong  to $G_{1/3}$ (see Figure~\ref{fig:2}).
\end{defin}

\begin{figure}
\begin{center}
\begin{tikzpicture}

\begin{pgfonlayer}{front}
\coordinate (x) at (0,0);
\fill (x) circle (2.8pt);

\foreach \i in {1,2,...,5}{
	\coordinate (v\i) at ($(0,9.5)+(234+12*\i:12)$);	
	\fill (v\i) circle (2.8pt);
	\node at ($(v\i)-(0,0.8)$) {$v_\i$};
}

\foreach \i in {1,2,3,4}{
	\draw[very thick] (v\i) -- (v\number\numexpr\i+1\relax);
}

\node at ($0.5*($(v1)+(v2)$)+(0,0.28)$) {\small$1/3$};
\node at ($0.5*($(v4)+(v5)$)+(0,0.28)$) {\small$1/3$};

\end{pgfonlayer}

\hedge{(x)}{(v2)}{(v1)}{6pt}{1.5}{}{black!40!white,opacity=0.3};
\hedge{(x)}{(v5)}{(v4)}{6pt}{1.5}{}{black!40!white,opacity=0.3};
\hedge{(x)}{(v3)}{(v2)}{7.4pt}{1.5}{}{black!40!white,opacity=0.3};
\hedge{(x)}{(v4)}{(v3)}{7.4pt}{1.5}{}{black!40!white,opacity=0.3};

\hedge{(v1)}{(v3)}{(v2)}{12pt}{1.5}{green!60!black}{green!80!black,opacity=0.3};
\hedge{(v3)}{(v5)}{(v4)}{12pt}{1.5}{green!60!black}{green!80!black,opacity=0.3};
\hedge{(v2)}{(v4)}{(v3)}{9.8pt}{1.5}{blue!40!black}{blue!60!black,opacity=0.3};

\end{tikzpicture}
\end{center}
\caption{An $x$-absorber of order $5$ on $v_1,\dots,v_5$ with $v_1v_2$ and $v_4v_5$ from $G_{1/3}$.}
\label{fig:2}
\end{figure}
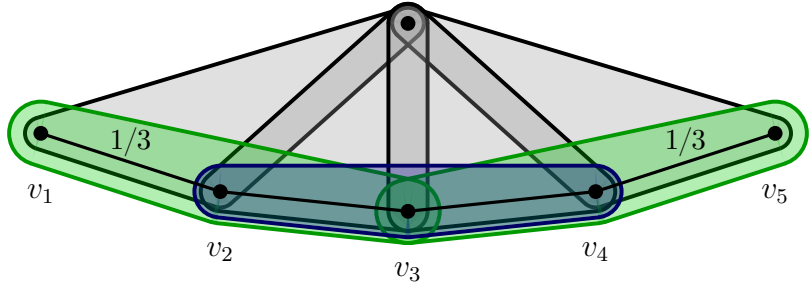

The Absorbing Lemma states that every 3-graph $H$ with sufficiently large minimum degree contains a relatively short absorbing path which may absorb  a small fraction of the vertices of $H$.

\begin{lemma}[Absorbing Lemma]\label{abs59} There exists $\gamma_0>0$ and $n_{\ref{abs59}}$ such that for every $0<\gamma<\gamma_0$ and every $n\ge n_{\ref{abs59}}$  the following is true. If $\delta(H)\ge.8\binom{n-1}2$, then
there exists in $H$ a $\gamma^{2}n$-absorbing path $A$ with $(1/3)$-large endpairs  and  with $|V(A)|\le\gamma n$.
\end{lemma}

We will build an absorbing  path $A$ by connecting, via Lemma \ref{conn},   
disjoint absorbers of order 4 or 5, by paths of length 12.
For the proof of Lemma~\ref{abs59} we need Claims~\ref{tx} and~\ref{456} stated below.

Recall the definition of the subhypergraph $H'\subseteq H$ in~\eqref{H0}.
Fix a vertex $x\in V$ and let~$T^x$ stand for the family of the vertex  sets of all triangles in the link graph $H'(x)$.

\begin{claim}\label{tx}
 There exists some $c>0$ such that for every $x\in V$ with $|V|=n$ sufficiently large we have
\begin{enumerate}[label=\alabel]
\item\label{it:a} $|T^x|\ge.296262\binom{n-1}3$,
\item\label{it:b} $|T^x\cap H'|\ge0.005\binom{n-1}3$, and
\item\label{it:c} $T^x\cap H'$ contains at least $cn^6$ copies of $K_{2,2,2}$.\end{enumerate}
\end{claim}

\begin{proof}
By the definition of  $H'\subseteq H$ and by
Claim \ref{f},
\begin{align}
\delta(H')
\ge\delta(H)-\binom{n-1-\delta(G_{1/3})}2
&\ge\delta(H)-\binom{n-1-g_{.8}(1/3)(n-1)}2\nonumber\\
&\ge .8\binom{n-1}2-\binom{.3(n-1)}2\label{H'H}
\ge.709\frac{n^2}2,\end{align}
for large~$n$.
 The link graph $H'(x)$ has at least $\delta(H')\ge.709n^2/2$ edges and by Lemma \ref{NS},
$$|T^x|\ge.709(2\cdot(.709)-1)\frac{n^3}6\ge.296262\binom{n}3.$$ As
$$|H'|\ge\delta(H')\frac{n}{3}\ge.709\frac{n^3}6\ge.709\binom n3,$$ we have

\begin{align*}|T^x\cap H'|\ge|T^x|+|H'|-|T^x\cup H'|\ge( .296262+.709-1)\binom {n}3> 0.005\binom {n}3.\end{align*}
Therefore, by Lemma \ref{erdos} with $d=0.005$ and $h=2$, there exists a constant $c$ such that for $n\ge n_{\ref{erdos}}$ there are at least $cn^6$ copies of $K_{2,2,2}$ in $T^x\cap H'$. 
\end{proof}

Recall that $e\in H'$ if $\binom e3\cap G_{1/3}\neq\emptyset$, and  $e\in T^x$ if the vertices of $e$ form a triangle in~$H'(x)$.
We will need the following notions.

\begin{defin}\label{friends}
 \noindent Let $x$ be a vertex of $H$.

 \begin{itemize}

 \item  Every copy of $K_{2,2,2}$ contained in the 3-graph $T^x\cap H'$  is  called \emph{$x$-friendly}.

 \item For a copy $K$ of $K_{2,2,2}$ in $H'$, let $S_K$ denote the set of all vertices $x\in V(H)$ for which $K$ is $x$-friendly.

 \end{itemize}
 \end{defin}

\begin{claim}\label{456}
 Every copy $K$ of $K_{2,2,2}$ in $H'$ contains a path of order 4 or 5, which is an $x$-absorber  for every $x\in S_K$.
\end{claim}

\begin{proof} Let the three partition classes of $K$ are $\{u_1,u_2\}$, $\{v_1,v_2\}$, and $\{w_1,w_2\}$. We have 
to find a path of order 4 or 5 in $K$  whose endpairs are in $G_{1/3}$. Consider the two disjoint hyperedges
$\{u_1,v_1,w_1\}$ and $\{u_2,v_2,w_2\}$ of $K$. By definition of $H'$ each of these two hyperedges must contain at least one pair from $G_{1/3}$. By symmetry, it suffices to consider the following two cases.

\noindent{\bf Case 1} ($\{u_1,v_1\}\in G_{1/3}$ and $\{u_2,v_2\}\in G_{1/3}$){\bf.} Then $u_2v_1w_1u_2v_2$ is an $x$-absorber of order 5 for all $x\in S_K$.

\noindent{\bf Case 2} ($\{u_1,v_1\}\in G_{1/3}$ and $\{u_2,w_2\}\in G_{1/3}$){\bf.} Then $u_1v_1w_2u_2$ is an $x$-absorber of order 4 for all $x\in S_K$.
\end{proof}

With these preparations we can prove the Absorbing Lemma.

\begin{proof}[Proof of Lemma~\ref{abs59}] 
Set
$$\gamma\le\gamma_0=\min\left\{\frac{c}{61},0.001\right\},$$
where $c$ is given by Claim \ref{tx}.
Let $\gamma\le \gamma_0$ and $n$ be sufficiently large.
Select randomly  a family ${\mathcal F}'$ of
copies  of $K_{2,2,2}$ in $H'$, independently and with probability $p=\tfrac1{30}\gamma n^{-5}$.
By Markov's inequality, with probability at least $1-0.5-0.4=0.1$,   ${\mathcal F}'$ satisfies:
\begin{itemize}
\item $|{\mathcal F}'|\le 2 n^6p=\frac{\gamma}{15} n$,
\item there are at most $\tfrac52\times n^6\times 6\times n^5p^2= 15n^{11}p^2$ pairs of vertex-intersecting copies of $K_{2,2,2}$ in ${\mathcal F}'$,
\end{itemize}
and, for large $n$, by Chernoff's inequality and using  Claim \ref{tx}, with probability greater than $0.9$,
\begin{itemize}
\item for every vertex $x$, there are at least
$\frac12 c n^6p$ $x$-friendly
 copies of   $K_{2,2,2}$  in  ${\mathcal F}'$.
\end{itemize}
Hence, there exists a family ${\mathcal F}'$ of copies of $K_{2,2,2}$ in $H'$ satisfying all three above conditions.
By removing from ${\mathcal F}'$  one
copy  of each intersecting pair, we obtain a subfamily $\mathcal F$ such that

\begin{itemize}
\item $|{\mathcal F}|\le\frac{\gamma}{15}n$,
\item  $\mathcal F$ consists of  disjoint copies of   $K_{2,2,2}$ in $H'$,
\item for every vertex $x$, there are at least
$$\tfrac12cn^6p- 15n^{11}p^2\ge \tfrac c{60}\gamma n-\tfrac{15}{900}\gamma^2 n\ge\gamma^2n$$
 $x$-friendly   copies of  $K_{2,2,2}$ in $\mathcal F$,
\end{itemize}
where for the last inequality we used the bound $\gamma\le c/61$.

Recall that by Claim \ref{456} each copy $K$ of $K_{2,2,2}$ in $H'$ contains a path of order 4 or 5 which is an $x$-absorber for all $x\in S_K$.
We now select one absorber  from each copy of $K_{2,2,2}$ in $\mathcal F$. Let us denote the resulting family of paths by $\mathcal P$ and note that $|\mathcal P|=|\mathcal F|\le \gamma n/15$.
Using Lemma \ref{conn}, we will connect all paths in $\mathcal P$ into one path $A$ of order at most
$$(5+10)|\mathcal P|-10\le 15\times \frac{\gamma}{15} n\le \gamma n.$$
(There are at most 5 vertices on a path in $\mathcal P$ and the number of new vertices connecting this path with another one is, by Lemma \ref{conn}, 14-4=10, since 4 of the 14 vertices belong to the paths in $\mathcal P$.)

Let $P_1,\dots,P_t$, $t\le\frac{\gamma}{15}n$, be the  paths (of order
4 or 5) in $\mathcal F$. Assume that, for some $i=1,\dots,t-1$, we have already connected
$P_1,\dots,P_i$ into a path $A_i$ of order at most $15i$. Let $H_i$ be the subhypergraph of $H$ obtained
by removing from $H$ all vertices of $A_i$ along with all vertices of $P_{i+1}\cup\dots\cup P_t$,
except for one endpair $e_{i+1}=w_{i+1}w'_{i+1}$ of $P_{i+1}$ and the endpair $e_i=w_iw'_i$ of $P_i$ (which is also an
endpair of $A_i$). Since $\gamma\le\gamma_0\le0.001$, assuming $n_{\ref{abs59}}\ge n_{\ref{conn}}/(1-\gamma_0)$,
$$|V(H_i)|\ge n-15t\ge (1-\gamma_0)n\ge n_{\ref{conn}},$$
$$\delta(H_i)\ge \delta(H)-15tn\ge\delta(H)-\gamma n^2\ge.799\binom{n-1}2\ge.799\binom{|V(H_i)|-1}2,$$
 and
$$\min\{\deg_{H_i}(w_i,w'_i),\deg_{H_i}(w_{i+1},w'_{i+1})\}\ge\frac13n-15tn\ge.33n\ge.33|V(H_i)|.$$ Hence, the assumptions of Lemma \ref{conn} are satisfied, and so there is a path
$Q_i$ of length 12 connecting $e_i$ and $e_{i+1}$ in $H_i$. The concatenation of the paths $A_i$,
$Q_i$, and $P_{i+1}$ constitutes the path $A_{i+1}$. Finally, set $A=A_t$.

To see that $A$ is indeed a $\gamma^2n$-absorbing path  in $H$, consider an arbitrary subset $U\subseteq
V\setminus V(A)$ of size $|U|\le\gamma^2n$. Since for every $x\in U$ there are at least $\gamma^2n$
$x$-absorbers $P_i$ in $A$, there is a one-to-one mapping $f\colon U\to\{1,\dots,t\}$ such
that for every $x\in U$, $P_{f(x)}$ is an $x$-absorber. Let $(v_1^x,\dots,v_i^x)$ be the vertices
of the path $P_{f(x)}$ ($4\le i\le 5$). Then
the path obtained from $A$ by replacing, for each $x\in U$, the edges $\{v_1^x,v_2^x,v_3^x\}$ and
$\{v_2^x,v_3^x,v_4^x\}$ with  $\{v_1^x,v_2^x,x\}$, $\{v_2^x,x,v_3^x\}$, and $\{x,v_3^x,v_4^x\}$, is
the desired path $A_U$. 
\end{proof}

\subsection{The Reservoir Lemma}

The next preparatory step toward the proof of Theorem~\ref{main} is to put aside  a reservoir
set $R$ which should be small,  quickly reachable from any
 pair in~$G_{1/3}$, and, moreover, the induced subhypergraph $H[R]$ should satisfy the assumption of Lemma~\ref{conn} with some margin. We state this lemma in a general form.

\begin{lemma}\label{res}  Let $U_1,\dots,U_s$ be subsets of an $n$-element set $V$ and let $L_1,\dots, L_g$ be graphs on $V$, where $s$ and $g$ are both polynomials in $n$ and such that for constants $\alpha_i$, $\beta_j\in(0,1)$ for $i=1,\dots,s$ and $j=1,\dots,g$
we have $|U_i|\ge \alpha_i n$  and $|L_j|\ge \beta_j \binom{n}2$, $j=1,\dots,g$. 

Then for every constant $p$, $0<p<1$ there is $n_{\ref{res}}=n_{\ref{res}}(p)$ such that if $n\ge n_{\ref{res}}$ then there exists a subset $R\subset V$ satisfying
\begin{enumerate}[label=\alabel]
\item\label{ra} $\big||R|-pn\big|\le pn^{2/3}$,

\item\label{rb} for all $i=1,\dots,s$, we have $|U_i\cap R|\ge (\alpha_i-2n^{-1/3})|R|$, and

\item\label{rc} for all $i=1,\dots,g$, we have $|L_j[R]|\ge (\beta_j-3n^{-1/3})\binom{|R|}2$.
\end{enumerate}
\end{lemma}

\begin{proof} Select a binomial random subset $R$ of  $V$ by
including to $R$ every element of $V$, independently, with probability $p$. The random
variable $|R|$ has the binomial distribution with expectation $np$.
By Chebyshev's inequality,  with probability tending to 1 as $n\to\infty$, part~\ref{ra} holds.

For every $i$, the random variable $X_i=|U_i\cap R|$ is also binomially
distributed, with expectation
 $$\EE[X_i]=|U_i|p\ge\alpha_i np.$$
 Thus, by a standard
application of Chernoff's bound (see \cite{JLR}*{inequality (2.6)}) we have
$$\PP\left(\exists i\colon X_i\le np(\alpha_i-n^{-1/3})\right)
\le \PP\left(\exists i\colon X_i\le |U_i|p-pn^{2/3}\right)
\le s\cdot\exp(-(p/2)n^{1/3})=o(1),$$
where in the last step we used, in passing, the trivial bound $|U_i|\le n$.
Hence, using the estimate in~\ref{ra}, for large $n$, we have
\begin{align*}\PP\left(\exists i\colon X_i\!\le\!(\alpha_i-2n^{-1/3})|R|\right)&=o(1)+\PP\left(\exists i\colon X_i\le(\alpha_i-2n^{-1/3})|R|,\, |R|\le np(1+n^{-1/3})\right)\\&
\le o(1)+\PP\left(\exists i\colon X_i\le(\alpha_i-2n^{-1/3}) np(1+n^{-1/3})\right)\\&\le o(1)+\PP\left(\exists i\colon X_i\le np(\alpha_i-n^{-1/3})\right)\\&=o(1).\end{align*}
Consequently, the randomly chosen set $R$ satisfies
 condition~\ref{rb} with probability tending to 1 as~$n\to\infty$.

 For part~\ref{rc}, fix  $i$ and consider a random variable $Y_i=|L_i[R]|$ counting the number of
 edges $\{u,w\}\in L_i$ with $\{u,w\}\subseteq
 R$. Note that
$$\EE[Y_i]=|L_i|p^2\ge\beta_i\binom {n}2p^2.$$
We apply to $Y_i$ Janson's inequality (see, e.g., \cite{JLR}*{Theorem~2.14}), which states that
$$\PP\left(Y_i\le \EE[Y_i]-t\right)\le\exp\{-t^2/\overline\Delta\}.$$
Here $\overline\Delta=\sum\sum \EE[I_eI_f]$, where the summation runs over all ordered pairs of not necessarily distinct edges of $L_i$ which share at least one vertex, while $I_e=1$ when $e\subset R$ and $I_e=0$ otherwise.
Observe that, up to the order of magnitude,  $\overline\Delta$ is equal to the expected number of pairs of edges of $L_i$, sharing a vertex, whose all three vertices are included in  $R$.
 Thus, $\overline\Delta=\Theta(n^3)$, and, consequently, with $t=n^{-1/3}\binom {n}2p^2$, we have
$$\PP\left(\exists i \colon Y_i\le(\beta_i-n^{-1/3})\binom {n}2p^2\right)\le \PP\left(\exists i \colon Y_i\le(\EE[Y_i]-t\right)  \le g \exp\{-\Theta(n^{1/3})\}=o(1).$$
Using part~\ref{ra} again,  for large $n$, we obtain
\begin{align*}
&\PP\left(\exists i\colon Y_i\!\le\!(\beta_i-3n^{-1/3})\binom X2\right)\\
&\hspace{2.5cm}=o(1)+\PP\left(\exists i\colon Y_i\le(\beta_i-3n^{-1/3})\binom X2,\, X\le np(1+n^{-1/3})\right)\\
&\hspace{2.5cm}\le o(1)+\PP\left(\exists i\colon Y_i\le(\beta_i-3n^{-1/3})\binom n2 p^2(1+n^{-1/3})^2\right)\\
&\hspace{2.5cm}\le o(1)+\PP\left(\exists i\colon Y_i\le(\beta_i-n^{-1/3})\binom n2p^2\right)\\
&\hspace{2.5cm}=o(1),\end{align*}
which means that the random set $R$  satisfies the condition of part ~\ref{rc} with probability tending to 1 as $n\to\infty$. In summary, for sufficiently large $n$, the probability that at least one of conditions~\ref{ra},~\ref{rb}, or~\ref{rc}
fails is less than 1, and thus, there exists
a set $R\subset V$  satisfying  all three properties~\ref{ra},~\ref{rb}, and~\ref{rc}.
\end{proof}

\subsection{The  Cover Lemma}\label{PCL}

Recall that $K_{L,L,L}$ denotes the complete 3-partite 3-graph on vertex classes of size $L$.

\begin{lemma}[Cover Lemma]\label{clic}
For every $\rho>0$, $\lambda>0$, and an integer $L$, there exists an integer $n_{\ref{clic}}$  such that every 3-graph $H$ with
$n\ge n_{\ref{clic}}$ vertices and $\delta(H)\ge\left(\tfrac59+\lambda\right)\binom{n-1}2$, contains
 a family of vertex-disjoint copies of $K_{L,L,L}$, which together cover at least
$(1-\rho)n$ vertices of $H$.
\end{lemma}

In the proof  of Lemma \ref{clic} we will need the following result from~\cite{hps}.
\begin{theorem}\label{sch}
For every $\beta>0$ there exists $t_1$ such that every $3$-graph $H$ with  $t\ge t_1$ vertices,
$3|t$, and with $\delta(H)\ge(\tfrac59+\beta)\binom{t-1}2$ contains a perfect matching.
\end{theorem}

The proof of Lemma \ref{clic} consists of several short steps.
 We begin by applying the Weak Regularity Lemma (Lemma \ref{WRL}) to~$H$.
Let
\begin{equation}\label{epst0}\eps=\min\left\{\frac14\rho^2,\frac1{400}\lambda^2\right\},\qquad t_0\ge\max\left\{\frac{13}{\lambda},2t_1\right\}\qquad n_{\ref{clic}}\ge\max\left\{n_{\ref{WRL}}, T_0n_{\ref{erdos}}\right\},
\end{equation}
where $T_0$ is given by Lemma \ref{WRL}.
 We apply Lemma \ref{WRL} to $H$, obtaining an $\eps$-regular partition $(V_1,\dots,V_t)$, where $t_0\le t\le T_0$.
Let us call the sets $V_i$  \emph{clusters} and below we consider the {\it
  cluster $3$-graph}~$K=K(\lambda/12,\eps)$ on the vertex set $[t]=\{1,\dots,t\}$. First, we define two auxiliary 3-graphs on~$[t]$:
 \begin{itemize}
\item $D(\lambda/12)$
consisting of all triples $\{i_1,i_2,i_3\}\subset[t]$ such that $d_H(V_{i_1},V_{i_2},V_{i_3})\ge\lambda/12$, and
\item $R(\varepsilon)$
 consisting of all triples $\{i_1,i_2,i_3\}\subset[t]$ such that $H[V_{i_1},V_{i_2},V_{i_3}]$ is
  $\varepsilon$-regular.
\end{itemize}

Having defined $D$ and $R$ we define $K=K(\lambda/12,\eps)$ as the intersection
\begin{equation}\label{cup}
K=D(\lambda/12)\cap R(\varepsilon).
\end{equation}
So, the  edges of $K$ are all triples of indices $\{i_1,i_2,i_3\}$
 such that
 $d_H(V_{i_1},V_{i_2},V_{i_3})\ge\lambda/12$ and $H[V_{i_1},V_{i_2},V_{i_3}]$ is $\varepsilon$-regular.

\begin{claim}\label{69}
$$\delta(D)\ge\left(\frac59+\frac23\lambda\right)\frac{t^2}2.$$
\end{claim}
\begin{proof}  Assume for simplicity that $t|n$ so that $|V_j|=n/t$ for every $j$. Fix $i\in[t]$ and set $N_i$ for the number of edges in $\bigcup_{j,l}|H[V_i,V_j,V_l]|$, where the union runs over all pairs $j,l$ such that $\{i,j,l\}\in D$. Then, on one hand,
$$N_i\le\deg_D(i)\cdot(n/t)^3,$$ while on the other hand, we can bound $N_i$ from below as follows.
We have $$\sum_{u\in V_i}\deg_H(u)\ge \left(\frac59+\lambda\right)(n/t)\binom{n-1}2,$$
but this sum counts the edges within $V_i$ three times and the edges with two vertices in $V_i$ twice.
Thus, the difference
$$\left(\frac59+\lambda\right)(n/t)\binom{n-1}2-3\binom{n/t}3-2\binom{n/t}2(t-1)(n/t)$$ sets a lower bound on the number of edges of $H$ with exactly one vertex in $V_i$. To get a lower bound on $N_i$ we need to further exclude  the at most $\binom{t-1}2(\lambda/12)(n/t)^3$ edges belonging to sub-3-graphs $H[V_i,V_j,V_l]$ with $\{i,j,l\}\not\in D$, as well as, the at most $\binom{n/t}2(t-1)(n/t)$ edges with two vertices in the same set $V_j$, $j\neq i$. Altogether, we arrive at the inequality
\begin{align*}N_i&\ge\left(\frac59+\lambda\right)(n/t)\binom{n-1}2-\frac{\lambda}{12}\binom{t-1}2(n/t)^3-3\binom{n/t}3-3\binom{n/t}2(t-1)(n/t),\end{align*}
whose right-hand side, for large $n$ and using the bound on $t_0$, can be further bounded from below by
\begin{align*}\left(\frac59+\lambda-\frac{\lambda}{12}-\frac1{t^2}-\frac3t\right)\frac{n^3}{2t}+O(n^2) 
&\ge \left(\frac59+\frac{11\lambda}{12}-\frac{\lambda^2}{169}-\frac{3\lambda}{13}+o(1)\right)\frac{n^3}{2t}\\
&\ge \left(\frac59+\frac23\lambda\right)\frac{n^3}{2t}.
\end{align*}
 Comparing the upper and  lower bound on $N_i$, we  obtain the desired estimate. 
\end{proof}

The just established  lower bound on the minimum degree in $D$ is essentially valid for the 3-graph $K$ as well, and thus, allows one to find in $K$  an almost perfect matching.
\begin{claim}\label{PM} There exists a matching $M$ in  $K$ with $|V(M)|\ge(1-\sqrt\varepsilon)t$.
\end{claim}

\begin{proof} We will find a sub-3-graph $K'$ of $K$ with $|V(K')|:=t'\ge(1-\sqrt\varepsilon)t$ and
\begin{equation}\label{boundt'}
\delta(K')\ge\left(\frac59+\frac12\lambda\right)\binom{t'-1}2.
\end{equation}
Once we are done with this task,  the claim will follow from Theorem \ref{sch} with $\beta=\tfrac12\lambda$ (note that, by (\ref{epst0}), $\eps\le 1/2$ and  $t'\ge t/2\ge t_0/2\ge t_1$).

Since the number of $\eps$-irregular triples is less than $\eps\binom t3$, the set $W$ of vertices $i\in[t]$ incident in $D$ to more than $\sqrt\eps\binom{t-1}2$ of them has size $|W|\le \sqrt\eps t$. For every vertex $i\in[t]\setminus W$,
$$\deg_K(i)\ge\left(\frac59+\frac23\lambda\right)\frac{t^2}2-\sqrt\eps\binom{t-1}2.$$
Let $W'\supset W$ be such that $t-|W'|$ is divisible by 3 and $|W'|\le|W|+2$.
As for every $i\in[t]\setminus W'$ and $j\in W'$ there are at most $t-2$ edges in $K$ containing both these vertices,
the induced sub-3-graph $K'=K-W'$ has minimum degree at least
$$\left(\frac59+\frac23\lambda\right)\frac{t^2}2-\sqrt\eps\binom{t-1}2-\left(\sqrt \eps t+2\right)(t-2)\ge \left(\frac59+\frac\lambda2\right)\binom{t-1}2,$$
where the first inequality follows from the bound $\eps\le \lambda^2/400$. Since $t\ge t'$, we have also (\ref{boundt'}) which, as explained above, completes the proof of Claim~\ref{PM}. 
\end{proof}

\begin{claim}\label{reg_cov} For every $\{i,j,l\}\in M$ the $3$-graph 
$H[V_i,V_j,V_l]$ contains a family ${\cQ}_{ijl}$ of vertex-disjoint copies of $K_{L,L,L}$ such that  $|{\cQ}_{ijl}|L\ge(1-\varepsilon) n/t$.
\end{claim}

\begin{proof} The claim will follow if we show that for all $W_k\subset V_k$, $|W_i|=\eps n/t$, $k=i,j,l$, the induced subhypergraph $H[W_i,W_j,W_l]$ contains a $K_{L,L,L}$.
Indeed, then a maximal family of vertex disjoint copies of $K_{L,L,L}$ in $H[V_i,V_j,V_l]$ may miss only less than $\eps n/t$ vertices in each set $V_k$ for $k=i,j,l$.

By $\eps$-regularity of $H[V_i,V_j,V_l]$ we have
$$|H[W_i,W_j,W_l]|\ge (\lambda/12-\eps)(n/t)^3=\frac1{27}(\lambda/12-\eps)(3n/t)^3.$$ Recalling that $t\le T_0$, $n\ge n_{\ref{clic}}$, and in view of (\ref{epst0}),  observe that
$$n/t\ge n/T_0\ge n_{\ref{clic}}/T_0\ge n_{\ref{erdos}}.$$
We apply Lemma \ref{erdos} to $H[W_i,W_j,W_l]$
 with $d= \tfrac1{27}(\lambda/12-\eps)$ and $h=L$, and conclude that there is in $H[W_i,W_j,W_l]$ a copy of $K_{L,L,L}$ (in fact, as many as $c(3n/t)^{3L}>0$, for some $c>0$).
 \end{proof}

\begin{proof}[Proof of Lemma~\ref{clic} (conclusion)]
Consider the union of all the families guaranteed by Claim \ref{reg_cov},  ${\cQ}=\bigcup_{\{i,j,l\}\in M}{\cQ}_{ijl}$. Since, clearly, $|M|\le t/3$ and, by Claim \ref{PM}, at most $\sqrt\varepsilon t(n/t)=\sqrt\eps n$ vertices of $H$ are not
covered by the clusters of $M$, we  conclude that $\cQ$ covers all
 but at most
 $$|M|\times 3\varepsilon n/t+\sqrt\varepsilon n\le\left(\varepsilon+\sqrt\varepsilon\right)n\le\rho n$$
vertices of $H$, where the last inequality follows from the assumption that $\eps\le\rho^2/4$. This completes  the proof of Lemma \ref{clic}.
\end{proof}

\begin{remark}\label{rem} As it will become clear in the next section, for our purposes it would be sufficient to prove Lemma \ref{clic} under the stronger assumption $\delta(H)\ge.7\binom{n-1}2$ and then the proof of Claim \ref{PM} would be quite straightforward, in particular, we would not need Theorem \ref{sch}.
But with Theorem \ref{sch} at hand, the strengthening of Lemma \ref{clic} comes for free and may be useful in the future work towards Conjecture~\ref{con59}.
\end{remark}

\section{Proof of Theorem \ref{main}}\label{final}
\label{sec:pfmain}
Let $H$ be a 3-graph with $\delta(H)\ge.8\binom{n-1}2$ and $n\ge n_{\ref{main}}$.
 To find a Hamiltonian cycle in $H$ we follow the five step outline presented in Section 2.1., and use Lemmas \ref{abs59}, \ref{res}, \ref{clic}, and \ref{conn} along the way. To facilitate their application, we begin with setting up the constants.

Let $\gamma_0$ be given by Lemma \ref{abs59} and let
\begin{equation}\label{ga}
\gamma=\min\{\gamma_0, 10^{-6}/3\}.
\end{equation}
 Further, let $n_{\ref{res}}=n_{\ref{res}}(\gamma^2/3)$ and  $n_{\ref{clic}}$ come from Lemma \ref{clic} with $\rho=\gamma^3$, $L=\lceil\tfrac13\gamma^{-3}\rceil$, and, say, $\lambda=1/9$.
 Finally, set
\begin{equation}\label{en}
n_{\ref{main}}=\max\left(\frac{n_{\ref{conn}}}{\gamma^2/4-14\gamma^3},\; n_{\ref{abs59}},\; \frac{n_{\ref{res}}}{1-\gamma},\; \frac{n_{\ref{clic}}}{1-\gamma-\gamma^2/2},\; 10^{12},\;\frac2{\gamma^3}\right).
\end{equation}

\subsection{Finding an absorbing path \texorpdfstring{$A$}{A} in \texorpdfstring{$H$}{H}}
By Lemma \ref{abs59} there exists in $H$ a $\gamma^2n$-absorbing path $A$ with $1/3$-large endpairs and with $|V(A)|\le\gamma n$. Recall that the powerful absorbing property of $A$ asserts that for every subset
$U\subset V(H)\setminus V(A)$ of size  $|U|\le \gamma^2n$
 there is a path $A_U$ in $H$  with
$V(A_U)=V(A)\cup U$ and with the same  endpairs as $A$ (see Definition~\ref{AP}). We are going to use this property at the very end of the proof.

\subsection{Finding a reservoir set \texorpdfstring{$R$}{R} in \texorpdfstring{$H- V(A)$}{H-V(A)}}

We need a small reservoir set $R$ which can be quickly reached from any $1/3$-large pair of vertices in $H$ and which satisfies the assumptions of Lemma \ref{conn}. The next claim is a simple corollary of Lemma \ref{res}.

\begin{claim}\label{res1} There exists a set $R\subset V(H)\setminus V(A)$ such that
\begin{enumerate}[label=\rmlabel]
\item[\upshape{(}$a$\upshape{)}] $\gamma^2n/4\le|R|\le \gamma^2n/2$,
\item[\upshape{(}$b_1$\upshape{)}] every $1/3$-large pair $e$ of $H$ has at least $.333|R|$ neighbours in $R$,
\item[\upshape{(}$b_2$\upshape{)}] every vertex $v$ of $H$ has at least $.697|R|$ neighbours in $G_{1/3}$ which belong to $R$,
\item[\upshape{(}$c$\upshape{)}] $\delta(H[R])\ge .7994\binom{|R|}2$.
\end{enumerate}
\end{claim}

\begin{proof}
Note that $|V(H)\setminus V(A)|\ge(1-\gamma)n\ge n_{\ref{res}}$ and apply Lemma \ref{res} to $V(H)\setminus V(A)$ with $p=\gamma^2/3$ and with the following choice of sets $U_i$ and graphs $L_j$:
\begin{itemize}
\item[($b$)] The sets $U_i$ are
\begin{itemize}
\item[($b_1$)] the sets $N_H(e)\setminus V(A)$, over all $e\in G_{1/3}$,
\item[($b_2$)]  the sets $N_{G_{1/3}}(v)\setminus V(A)$, over all $v\in V$, 
\end{itemize}
\item[($c$)]  The graphs $L_j$  are the graphs $H(v)-V(A)$, $v\in V$, obtained from the link graphs by removing the vertices on $A$.
\end{itemize}

The corresponding coefficients $\alpha_i$ for sets in group ($b_1$) are $\alpha_e=1/3-\gamma$, while in group ($b_2$) they are, by Claim \ref{degdeg}, 
with $\alpha_v=h_{.8}(1/3)-\gamma=.7-\gamma$.
The coefficient $\beta_j$ for graphs (part~($c$)) are $\beta_v=.8-3\gamma$, because every vertex $v$ belongs to at most $|V(A)|(n-2)<3\gamma\binom{n-1}2$ edges intersecting $V(A)$.

In the short argument below we use the bounds on $\gamma$ and $n$ stemming from~\eqref{ga} and~\eqref{en}.
Because  $n\ge10^{12}\ge3 64$, part ($a$) follows from Lemma \ref{res}~\ref{ra}. Parts ($b_1$) and ($b_2$) follow from Lemma \ref{res}~\ref{rb}, because $\gamma\le0.001$ and $n^{-3}\le 0.001$. Finally, part ($c$) follows from Lemma \ref{res}~\ref{rc}, using $\gamma\le0.0001$  and $n^{-3}\le 0.0001$, and the relation $\delta(H[R])=\min_{v\in R}|H(v)[R]|$.
\end{proof}

\subsection{Finding a collection of disjoint paths \texorpdfstring{$\mathcal P$}{P}, covering most of the vertices}
Our goal here is to find a collection of disjoint paths in $H$ which are disjoint from $V(A)\cup R$, cover almost all vertices in $V(H)\setminus(V(A)\cup R)$, and, most importantly, have $1/3$-large endpairs. This last condition is needed in the next step of the proof where we connect all these paths together.

Recall that $H'$ is a sub-3-graph of $H$ consisting of all edges containing at least one $1/3$-large pair, that is, an edge of the graph $G_{1/3}$ (see (\ref{H0})).
We have already shown that $\delta(H')\ge.709n^2/2$ (see (\ref{H'H})). Setting
$$H'':=H'-(V(A)\cup R)$$
 and noting that, by (\ref{ga}), $\gamma<0.0005$, we thus have
$$\delta(H'')\ge.709n^2/2-|V(A)\cup R|n\ge.709n^2/2-2\gamma n^2\ge.708n^2/2\ge.708\binom{n-1}2.$$
By (\ref{en}), we also have $|V(H'')|\ge n_{\ref{clic}}$. We apply Lemma \ref{clic} to $H'':=H'-(V(A)\cup R)$ with $\rho=\gamma^3$, $L=\lceil\tfrac13\gamma^{-3}\rceil$, and, say, $\lambda=1/9$, obtaining a family $\mathcal Q$ of vertex-disjoint copies of $K_{L,L,L}$ which together cover at least $(1-\gamma^3)|V(H'')|$ vertices of $H''$.

By the definition of $H'$, every copy $Q\in \mathcal Q$ of $K_{L,L,L}$ contains a path $P$ of length at least $3L-1$ with both endpairs in $G_{1/3}$.
Let $\mathcal P$ be the family of all these paths. Note that
$$|\mathcal P|=|\mathcal Q|\le n/3L<\gamma^3n.$$

Unlike in \cite{1112}, the path cover $\mathcal P$ we have gotten  consists of $\Theta(n)$ paths of length $O(1)$. Yet,  the number of these paths, $|\mathcal P|$,  is much smaller than $|R|$ (compare the above bound with the lower bound in part ($a$) of Claim \ref{res1}.) This  allows us to glue them all together using the reservoir set.

\subsection{Connecting the paths in \texorpdfstring{$\mathcal P$}{P} and the absorbing path \texorpdfstring{$A$}{A} into a long cycle}
 Our task is to connect all the paths in $\mathcal P$, as well as the absorbing path $A$, into one cycle. Let $m=|\mathcal P|+1$ be the number of these paths and let  $\mathcal P=\{P_1,\dots,P_{m-1}\}$.  Further, let $e_i$ and $f_{i+1}$ be the endpairs of $P_i$, $i=1,\dots,m$, where we set $P_m=A$ and $f_{m+1}=f_1$  for convenience. Recall that all these endpairs are ordered pairs of vertices and, treated as unordered pairs, they belong to $G_{1/3}$.

Then, the following claim is all what we need. It states that all pairs $\{e_i,f_i\}$, $i=1,\dots, m$,  can be simultaneously connected by short, mutually vertex-disjoint paths whose all inner vertices (other than those in $e_i$ and $f_i$) belong to $R$. Of course, this connecting scheme results in a cycle $C$ in $H$ containing all paths $P_1,\dots,P_{m-1}$ and $A$ as sub-3-graphs.

\begin{claim}\label{connR} Let $\delta(H)\ge.8\binom{n-1}2$ and let $R$ satisfy properties ($a$)-($c$) of Claim \ref{res1}. Further, let $m$ be an integer, $m\le \gamma^3n+1$, and let $e_1,\dots,e_m$ and $f_1,\dots, f_m$ be disjoint ordered pairs in $V\setminus R$ which belong to $G_{1/3}$. Then,  there are disjoint  paths $\Pi_1,\dots, \Pi_m$, each of length 16, where, for $i=1,\dots, m$, $\Pi_i$ has endpairs $e_i$ and $f_i$ and $V(\Pi_i)\setminus(e_i\cup f_i)\subset R$.
\end{claim}

\begin{proof} We show by induction on $i$ that the paths $\Pi_1,\dots,\Pi_m$ exist and that $|V(\Pi_i)\cap R|=14$. Suppose that for some $0\le i\le m-1$ we have already found paths $\Pi_j$ for all $j=1,\dots, i$. Together these paths occupy $14i\le 14(m-1)\le 14\gamma^3 n$ vertices of $R$. Let $R_i$ be the set of all the remaining vertices of $R$. Thus, by part ($a$) of Claim \ref{res1},
\begin{equation}\label{RiR}
|R_i|\ge|R|- 14\gamma^3n\ge n_{\ref{conn}}.
\end{equation}
The properties ($a$)-($c$) of $R$ established in Claim \ref{res1} imply the following, a bit weaker, properties  of every subset $R'\subseteq R$, with $|R'|\ge |R|-15\gamma^3n$:

\begin{enumerate}\item[($b'_1$)] every $1/3$-large pair $e$ of $H$ has at least $.33|R'|$ neighbours in $R'$,
\item[($b'_2$)] every vertex $v$ of $H$ has at least $.69|R'|$ neighbours in $G_{1/3}$ which belong to $R'$,
\item[($c'$)] $\delta(H[R'])\ge .799\binom{|R'|-1}2$.
\end{enumerate}
Indeed, to see, for instance, that ($c'$) holds, observe that
$$0.7994\binom{|R|}2-15\gamma^3n(|R|-2)\ge0.799\binom{|R|-1}2$$
follows from $0.0004|R|\ge 30\gamma^3n$, which, in turn, follows by the lower bound in ($a$) and by~\eqref{ga}.

 We  shall use these properties to connect the pairs $e_{i+1}=(v,u)$ and  $f_{i+1}=(y,x)$. Since~$.33+.69>1$ and   $\{u,v\}\in G_{1/3}$, by ($b_1'$) and ($b_2'$) applied to $R'=R_i$, there exists a vertex $w\in R_i$ such that $uvw\in H$ and $vw\in G_{1/3}$. This, in turn, implies that there is also a vertex $w'\in R_i$ such that $vww'\in H$ and $ww'\in G_{1/3}$.  Note that
$$|R_i\setminus\{w,w'\}|\ge |R|-14\gamma^3n-2\ge|R|-15\gamma^3n.$$
Similarly, this time applying ($b_1'$) and ($b_2'$) to $R'=R_i\setminus\{w,w'\}$, we argue that there exist two other vertices in $R_i$, $z$ and $z'$, such that $xyz,yzz'\in H$, while $zz'\in G_{1/3}$.

By ($b_1'$) applied again to $R'=R_i$, both $ww'$ and $zz'$ are $.33$-large in $H[R_i]$, the 3-graph induced in $H$ by $R_i$. Also, by ($c'$), we have $\delta(H[R_i])\ge.799\binom{|R_i|-1}2$. Hence, recalling also (\ref{RiR}), we are in position to apply Lemma \ref{conn} to $H[R_i]$ and conclude that there is a path~$\pi_{i+1}$ of length 12 between $(w,w')$ and $(z,z')$. This path, together with the previously constructed four edges, forms a desired 18-vertex path $\Pi_{i+1}$ between $e_{i+1}$ and $f_{i+1}$ (see Figure~\ref{fig:3}).
\end{proof}

\begin{figure}
\begin{center}
\begin{tikzpicture} 

\foreach \x in {0,1,2,3}{
	\fill (1.6*\x,1.5) circle (2.8pt);
	\fill (1.6*\x,-1.5) circle (2.8pt);
}

\draw[very thick] (0,1.5) -- (3.2,1.5);
\draw[very thick] (0,-1.5) -- (3.2,-1.5);

\node at (0,1.15) {$u$};
\node at (1.6,1.15) {$v$};
\node at (3.2,1.15) {$w$};
\node at (4.8,1.22) {$w'$};

\node at (0,-1.15) {$x$};
\node at (1.6,-1.15) {$y$};
\node at (3.2,-1.15) {$z$};
\node at (4.8,-1.085) {$z'$};

\node at (0.8,1.9) {\small$1/3$};
\node at (2.4,1.9) {\small$1/3$};
\node at (4.0,1.9) {\small$1/3$};

\node at (0.8,-1.9) {\small$1/3$};
\node at (2.4,-1.9) {\small$1/3$};
\node at (4.0,-1.9) {\small$1/3$};

\node at (0.8,1.25) {$e_{i+1}$};
\node at (0.8,-1.22) {$f_{i+1}$};

\draw[very thick, decorate, decoration={snake,segment length=6,amplitude=3}] 
	(3.2,-1.5) -- (4.8,-1.5) to[out=-30,in=30, distance=50] (4.8,1.5) -- (3.2,1.5);

\node at (4.5,0.25) {path $\pi_{i+1}$};
\node at (4.5,-0.25) {of length $12$};

\draw[very thick] (4.5,0) ellipse (2 and 2.5);
\node at (6.2,2.0) {$R$};

\end{tikzpicture}
\end{center}
\caption{Path $\Pi_{i+1}$ of length $16$ from $e_{i+1}$ to $f_{i+1}$.}
\label{fig:3}
\end{figure}
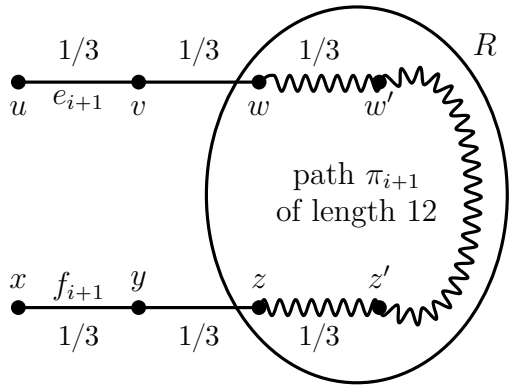

\subsection{Creating a Hamiltonian cycle in \texorpdfstring{$H$}{H}}
Let us denote by $T$ the set of vertices of $H''$, not covered by the paths in $\mathcal P$. It consists of up to $\gamma^3n$ vertices not covered by the copies of $K_{L,L,L}$ in $\mathcal Q$ plus up to $|\mathcal Q|$ vertices dropped from the each $Q_i$ whenever the path $P_i$ had $3L-1$ vertices and not all $3L$. Therefore,
$$|T|\le \gamma^3n+|\mathcal Q|\le 2\gamma^3n.$$
Observe further that, since $C\supset A$, we have $V\setminus V(C)\subset R\cup T$ and thus
$$|V\setminus V(C)|\le \frac12\gamma^2n+2\gamma^3n\le \gamma^2n.$$
Since the path $A$ forms a segment of $C$, we can employ the $\gamma^2n$-absorbing property of $A$ to the set $U:=V\setminus V(C)$. By replacing in $C$, the path $A$ by a path  $A_U$, we finally obtain a Hamiltonian cycle in $H$, which concludes the proof of Theorem~\ref{main}.\qed

\subsection*{Acknowledgement}
We thank the referee for her or his timely work und helpful comments.


\begin{bibdiv}
\begin{biblist}

\bib{berge}{book}{
   author={Berge, Claude},
   title={Graphs and hypergraphs},
   note={Translated from the French by Edward Minieka;
   North-Holland Mathematical Library, Vol. 6},
   publisher={North-Holland Publishing Co., Amsterdam-London; American
   Elsevier Publishing Co., Inc., New York},
   date={1973},
   pages={xiv+528},
   review={\MR{0357172 (50 \#9640)}},
}

\bib{Bermond}{article}{
   author={Bermond, J.-C.},
   author={Germa, A.},
   author={Heydemann, M.-C.},
   author={Sotteau, D.},
   title={Hypergraphes hamiltoniens},
   language={French, with English summary},
   conference={
      title={Probl\`emes combinatoires et th\'eorie des graphes},
      address={Colloq. Internat. CNRS, Univ. Orsay, Orsay},
      date={1976},
   },
   book={
      series={Colloq. Internat. CNRS},
      volume={260},
      publisher={CNRS, Paris},
   },
   date={1978},
   pages={39--43},
   review={\MR{539937 (80j:05093)}},
}

\bib{BB}{book}{
   author={Bollob{\'a}s, B{\'e}la},
   title={Extremal graph theory},
   note={Reprint of the 1978 original},
   publisher={Dover Publications, Inc., Mineola, NY},
   date={2004},
   pages={xx+488},
   isbn={0-486-43596-2},
   review={\MR{2078877 (2005b:05124)}},
}

\bib{Bus}{article}{
   author={Bu{\ss}, Enno},
   author={H{\`a}n, Hi{\d{\^e}}p},
   author={Schacht, Mathias},
   title={Minimum vertex degree conditions for loose Hamilton cycles in
   3-uniform hypergraphs},
   journal={J. Combin. Theory Ser. B},
   volume={103},
   date={2013},
   number={6},
   pages={658--678},
   issn={0095-8956},
   review={\MR{3127586}},
   doi={10.1016/j.jctb.2013.07.004},
}

\bib{Ch91}{article}{
   author={Chung, Fan R. K.},
   title={Regularity lemmas for hypergraphs and quasi-randomness},
   journal={Random Structures Algorithms},
   volume={2},
   date={1991},
   number={2},
   pages={241--252},
   issn={1042-9832},
   review={\MR{1099803}},
   doi={10.1002/rsa.3240020208},
}

\bib{dirac}{article}{
   author={Dirac, G. A.},
   title={Some theorems on abstract graphs},
   journal={Proc. London Math. Soc. (3)},
   volume={2},
   date={1952},
   pages={69--81},
   issn={0024-6115},
   review={\MR{0047308 (13,856e)}},
}

\bib{E}{article}{
   author={Erd{\H{o}}s, P.},
   title={On extremal problems of graphs and generalized graphs},
   journal={Israel J. Math.},
   volume={2},
   date={1964},
   pages={183--190},
   issn={0021-2172},
   review={\MR{0183654 (32 \#1134)}},
}

\bib{FR92}{article}{
   author={Frankl, P.},
   author={R{\"o}dl, V.},
   title={The uniformity lemma for hypergraphs},
   journal={Graphs Combin.},
   volume={8},
   date={1992},
   number={4},
   pages={309--312},
   issn={0911-0119},
   review={\MR{1204114}},
   doi={10.1007/BF02351586},
}

\bib{GPW}{article}{
   author={Glebov, Roman},
   author={Person, Yury},
   author={Weps, Wilma},
   title={On extremal hypergraphs for Hamiltonian cycles},
   journal={European J. Combin.},
   volume={33},
   date={2012},
   number={4},
   pages={544--555},
   issn={0195-6698},
   review={\MR{2864440}},
   doi={10.1016/j.ejc.2011.10.003},
}

\bib{Good}{article}{
   author={Goodman, A. W.},
   title={On sets of acquaintances and strangers at any party},
   journal={Amer. Math. Monthly},
   volume={66},
   date={1959},
   pages={778--783},
   issn={0002-9890},
   review={\MR{0107610 (21 \#6335)}},
}

\bib{hps}{article}{
   author={H{\`a}n, Hi{\d{\^e}}p},
   author={Person, Yury},
   author={Schacht, Mathias},
   title={On perfect matchings in uniform hypergraphs with large minimum
   vertex degree},
   journal={SIAM J. Discrete Math.},
   volume={23},
   date={2009},
   number={2},
   pages={732--748},
   issn={0895-4801},
   review={\MR{2496914 (2011a:05266)}},
   doi={10.1137/080729657},
}

\bib{hs}{article}{
   author={H{\`a}n, Hi{\d{\^e}}p},
   author={Schacht, Mathias},
   title={Dirac-type results for loose Hamilton cycles in uniform
   hypergraphs},
   journal={J. Combin. Theory Ser. B},
   volume={100},
   date={2010},
   number={3},
   pages={332--346},
   issn={0095-8956},
   review={\MR{2595675 (2011c:05181)}},
   doi={10.1016/j.jctb.2009.10.002},
}

\bib{YiJie}{article}{
   author={Han, Jie},
   author={Zhao, Yi},
   title={Minimum vertex degree threshold for loose Hamilton cycles in
   3-uniform hypergraphs},
   journal={J. Combin. Theory Ser. B},
   volume={114},
   date={2015},
   pages={70--96},
   issn={0095-8956},
   review={\MR{3354291}},
   doi={10.1016/j.jctb.2015.03.007},
}

\bib{YiJieCounter}{article}{
   author={Han, Jie},
   author={Zhao, Yi},
   title={Forbidding Hamilton cycles in uniform hypergraphs},
   eprint={1508.05623},
   note={Submitted.}
}

\bib{JLR}{book}{
   author={Janson, Svante},
   author={{\L}uczak, Tomasz},
   author={Rucinski, Andrzej},
   title={Random graphs},
   series={Wiley-Interscience Series in Discrete Mathematics and
   Optimization},
   publisher={Wiley-Interscience, New York},
   date={2000},
   pages={xii+333},
   isbn={0-471-17541-2},
   review={\MR{1782847 (2001k:05180)}},
   doi={10.1002/9781118032718},
}

\bib{KK}{article}{
   author={Katona, Gy. Y.},
   author={Kierstead, H. A.},
   title={Hamiltonian chains in hypergraphs},
   journal={J. Graph Theory},
   volume={30},
   date={1999},
   number={3},
   pages={205--212},
   issn={0364-9024},
   review={\MR{1671170 (99k:05124)}},
   doi={10.1002/(SICI)1097-0118(199903)30:3$<$205::AID-JGT5$>$3.3.CO;2-F},
}

\bib{KKMO}{article}{
   author={Keevash, Peter},
   author={K{\"u}hn, Daniela},
   author={Mycroft, Richard},
   author={Osthus, Deryk},
   title={Loose Hamilton cycles in hypergraphs},
   journal={Discrete Math.},
   volume={311},
   date={2011},
   number={7},
   pages={544--559},
   issn={0012-365X},
   review={\MR{2765622 (2012d:05274)}},
   doi={10.1016/j.disc.2010.11.013},
}

\bib{khan}{article}{
   author={Khan, Imdadullah},
   title={Perfect matchings in 3-uniform hypergraphs with large vertex
   degree},
   journal={SIAM J. Discrete Math.},
   volume={27},
   date={2013},
   number={2},
   pages={1021--1039},
   issn={0895-4801},
   review={\MR{3061464}},
   doi={10.1137/10080796X},
}

\bib{kmo}{article}{
   author={K{\"u}hn, Daniela},
   author={Mycroft, Richard},
   author={Osthus, Deryk},
   title={Hamilton $\ell$-cycles in uniform hypergraphs},
   journal={J. Combin. Theory Ser. A},
   volume={117},
   date={2010},
   number={7},
   pages={910--927},
   issn={0097-3165},
   review={\MR{2652102 (2011g:05216)}},
   doi={10.1016/j.jcta.2010.02.010},
}

\bib{kob}{article}{
   author={K{\"u}hn, Daniela},
   author={Osthus, Deryk},
   title={Loose Hamilton cycles in 3-uniform hypergraphs of high minimum
   degree},
   journal={J. Combin. Theory Ser. B},
   volume={96},
   date={2006},
   number={6},
   pages={767--821},
   issn={0095-8956},
   review={\MR{2274077 (2007h:05115)}},
   doi={10.1016/j.jctb.2006.02.004},
}

\bib{kot}{article}{
   author={K{\"u}hn, Daniela},
   author={Osthus, Deryk},
   author={Treglown, Andrew},
   title={Matchings in 3-uniform hypergraphs},
   journal={J. Combin. Theory Ser. B},
   volume={103},
   date={2013},
   number={2},
   pages={291--305},
   issn={0095-8956},
   review={\MR{3018071}},
   doi={10.1016/j.jctb.2012.11.005},
}

\bib{klas}{article}{
   author={Markstr{\"o}m, Klas},
   author={Ruci{\'n}ski, Andrzej},
   title={Perfect matchings (and Hamilton cycles) in hypergraphs with large
   degrees},
   journal={European J. Combin.},
   volume={32},
   date={2011},
   number={5},
   pages={677--687},
   issn={0195-6698},
   review={\MR{2788783 (2012c:05247)}},
   doi={10.1016/j.ejc.2011.02.001},
}

\bib{MM}{article}{
   author={Moon, J. W.},
   author={Moser, L.},
   title={On a problem of Tur\'an},
   language={English, with Russian summary},
   journal={Magyar Tud. Akad. Mat. Kutat\'o Int. K\"ozl.},
   volume={7},
   date={1962},
   pages={283--286},
   review={\MR{0151955 (27 \#1936)}},
}

\bib{ns}{article}{
   author={Nordhaus, E. A.},
   author={Stewart, B. M.},
   title={Triangles in an ordinary graph},
   journal={Canad. J. Math.},
   volume={15},
   date={1963},
   pages={33--41},
   issn={0008-414X},
   review={\MR{0151957 (27 \#1938)}},
}

\bib{raz}{article}{
   author={Razborov, Alexander A.},
   title={On the minimal density of triangles in graphs},
   journal={Combin. Probab. Comput.},
   volume={17},
   date={2008},
   number={4},
   pages={603--618},
   issn={0963-5483},
   review={\MR{2433944 (2009i:05118)}},
   doi={10.1017/S0963548308009085},
}

\bib{sur}{article}{
   author={R{\"o}dl, Vojt{\v{e}}ch},
   author={Ruci{\'n}ski, Andrzej},
   title={Dirac-type questions for hypergraphs---a survey (or more problems
   for Endre to solve)},
   conference={
      title={An irregular mind},
   },
   book={
      series={Bolyai Soc. Math. Stud.},
      volume={21},
      publisher={J\'anos Bolyai Math. Soc., Budapest},
   },
   date={2010},
   pages={561--590},
   review={\MR{2815614 (2012j:05008)}},
   doi={10.1007/978-3-642-14444-8\_16},
}

\bib{1112}{article}{
   author={R{\"o}dl, Vojt{\v{e}}ch},
   author={Ruci{\'n}ski, Andrzej},
   title={Families of triples with high minimum degree are Hamiltonian},
   journal={Discuss. Math. Graph Theory},
   volume={34},
   date={2014},
   number={2},
   pages={361--381},
   issn={1234-3099},
   review={\MR{3194042}},
   doi={10.7151/dmgt.1743},
}

\bib{rrs3}{article}{
   author={R{\"o}dl, Vojt{\v{e}}ch},
   author={Ruci{\'n}ski, Andrzej},
   author={Szemer{\'e}di, Endre},
   title={A Dirac-type theorem for 3-uniform hypergraphs},
   journal={Combin. Probab. Comput.},
   volume={15},
   date={2006},
   number={1-2},
   pages={229--251},
   issn={0963-5483},
   review={\MR{2195584 (2006j:05144)}},
   doi={10.1017/S0963548305007042},
}
		
\bib{k}{article}{
   author={R{\"o}dl, Vojt{\v{e}}ch},
   author={Ruci{\'n}ski, Andrzej},
   author={Szemer{\'e}di, Endre},
   title={An approximate Dirac-type theorem for $k$-uniform hypergraphs},
   journal={Combinatorica},
   volume={28},
   date={2008},
   number={2},
   pages={229--260},
   issn={0209-9683},
   review={\MR{2399020 (2009a:05146)}},
   doi={10.1007/s00493-008-2295-z},
}

\bib{rrs}{article}{
   author={R{\"o}dl, Vojt{\v{e}}ch},
   author={Ruci{\'n}ski, Andrzej},
   author={Szemer{\'e}di, Endre},
   title={Perfect matchings in large uniform hypergraphs with large minimum
   collective degree},
   journal={J. Combin. Theory Ser. A},
   volume={116},
   date={2009},
   number={3},
   pages={613--636},
   issn={0097-3165},
   review={\MR{2500161 (2010d:05124)}},
   doi={10.1016/j.jcta.2008.10.002},
}

\bib{3}{article}{
   author={R{\"o}dl, Vojt{\v{e}}ch},
   author={Ruci{\'n}ski, Andrzej},
   author={Szemer{\'e}di, Endre},
   title={Dirac-type conditions for Hamiltonian paths and cycles in
   3-uniform hypergraphs},
   journal={Adv. Math.},
   volume={227},
   date={2011},
   number={3},
   pages={1225--1299},
   issn={0001-8708},
   review={\MR{2799606 (2012d:05213)}},
   doi={10.1016/j.aim.2011.03.007},
}

\bib{St90}{thesis}{
   author={Steger, Angelika},
   title={Die Kleitman-Rothschild Methode},
   type={PhD thesis},
   organization={Forschungsinstitut f\"ur Diskrete Mathematik, Rheinische Friedrichs-Wilhelms-Universit\"at Bonn},
   date={1990},
}

\bib{Sz}{article}{
   author={Szemer{\'e}di, Endre},
   title={Regular partitions of graphs},
   language={English, with French summary},
   conference={
      title={Probl\`emes combinatoires et th\'eorie des graphes},
      address={Colloq. Internat. CNRS, Univ. Orsay, Orsay},
      date={1976},
   },
   book={
      series={Colloq. Internat. CNRS},
      volume={260},
      publisher={CNRS, Paris},
   },
   date={1978},
   pages={399--401},
   review={\MR{540024 (81i:05095)}},
}

\end{biblist}
\end{bibdiv}
\end{document}